\documentclass[a4paper,11pt]{amsart}

\usepackage{style}

\usepackage{graphicx}
\usepackage{eucal}
\usepackage{amssymb}
\usepackage{mathrsfs}
\usepackage{amssymb}
\usepackage{amsmath}
\usepackage{picture}
\usepackage{graphics}
\usepackage{mathrsfs}
\usepackage{amsmath,amssymb,amsthm,amscd,amsfonts} 
\usepackage{mathrsfs}
\usepackage{fancyhdr}
\usepackage{colonequals}
\usepackage[all]{xy}
\usepackage{stmaryrd}
\usepackage{color}
\usepackage{multicol}
\usepackage{mathtools}
\usepackage{hyperref}
\usepackage{colonequals}
\usepackage{tikz-cd}
\usepackage{dcpic,pictex}

\usepackage{algorithm2e}  
\usepackage{graphicx}
\usepackage{eucal}
\usepackage{graphics}
\usepackage{color,cite,graphicx}
\usepackage{tikz}
\definecolor{gray}{gray}{0.4}
\usepackage[all]{xy}
\usepackage{enumerate}
\usepackage{rotating}
\usepackage{caption}
\usepackage{yfonts}
\usepackage{multicol}

\DeclareMathOperator{\spin}{spin}

\DeclareMathOperator{\divi}{div}
\DeclareMathOperator{\OG6}{OG6}
\DeclareMathOperator{\og}{OG10}
\DeclareMathOperator{\Hom}{Hom}

\DeclareMathOperator{\Ker}{Ker}
\DeclareMathOperator{\Sgn}{sgn}

\DeclareMathOperator{\Sign}{sign}

\DeclareMathOperator{\NS}{NS}
\DeclareMathOperator{\T}{T}

\DeclareMathOperator{\rank}{rk}
\DeclareMathOperator{\id}{id}

\DeclareMathOperator{\rk}{rk}
\DeclareMathOperator{\Aut}{Aut}

\DeclareMathOperator{\Mon}{Mon}

\DeclareMathOperator{\bL}{\mathbf{L}}
\DeclareMathOperator{\bO}{\mathrm{O}}
\DeclareMathOperator{\A}{\mathbf{A}}
\DeclareMathOperator{\D}{\mathbf{D}}
\DeclareMathOperator{\E}{\mathbf{E}}
\DeclareMathOperator{\h}{\mathbf{H}}
\DeclareMathOperator{\K}{\mathbf{K}}
\DeclareMathOperator{\U}{\mathbf{U}}
\DeclareMathOperator{\bR}{\mathbf{R}}
\DeclareMathOperator{\bLambda}{\boldsymbol{\Lambda}}

\title[Nonsymplectic automorphisms of prime order on O'Grady's sixfolds]{Nonsymplectic automorphisms of prime order \\ on O'Grady's sixfolds}

\author{Annalisa Grossi}
\address{Fakultät für Mathematik, Technische Universität Chemnitz, Reichenhainer Str. 39, 09126 Chemnitz, Germany}
\email{annalisa.grossi@math.tu-chemnitz.de}

\urladdr{https://annalisagrossi92.wixsite.com/home}

\date{\today}

\makeatletter
\@namedef{subjclassname@2020}{%
	\textup{2020} Mathematics Subject Classification}
\makeatother



\usepackage{microtype}
\begin{document}

\maketitle

\begin{abstract}	
We classify nonsymplectic automorphisms of prime order on irreducible holomorphic symplectic manifolds of O'Grady's 6-dimensional deformation type. More precisely, we give a classification of the invariant and coinvariant sublattices of the second integral cohomology group.
		
\end{abstract}
\section{Introduction}\label{Introduction}
\subsection{Background}
  Irreducible holomorphic symplectic manifolds \(X\) are simply connected compact K\"{a}hler manifolds carrying a nowhere degenerate holomorphic symplectic form \(\sigma_X\) which spans \(H^{0}(X, \Omega_X^{2})\).
 
In dimension 2, irreducible holomorphic symplectic manifolds are K3 surfaces. Fujiki \cite{fujiki1983primitively} and Beauville \cite{beauville1983varietes} found examples in higher dimensions: more precisely the Hilbert scheme of \(n\) points on a K3 surface and the generalized Kummer manifold in the sense of Beauville \cite{beauville1983varietes} are irreducible holomorphic symplectic manifolds of dimension \(2n\). Manifolds which are deformation equivalent to the Hilbert scheme and to the generalized Kummer manifold are called manifolds of K3\(^{[n]}\) type and of Kum\(_n\) type respectively. 

Mukai \cite{mukai1984symplectic} discovered a symplectic form on moduli spaces of sheaves on symplectic surfaces assuming some conditions on them. However he proved that all nonsingular irreducible holomorphic symplectic manifolds obtained in this way were a deformation of known examples. The singular ones admit a resolution of singularities which is irreducible holomorphic symplectic only in two cases discovered by O'Grady: one in dimension 6 \cite{o2000new} and one in dimension 10 \cite{o1997desingularized}. Manifolds that are deformation equivalent to O'Grady's sixfold and to O'Grady's tenfold are called manifolds of \(\OG6\) type and manifolds of \(\og\) type respectively.

\subsection{Automorphisms of irreducible holomorphic symplectic manifolds}\label{Automorphisms of irreducible holomorphic symplectic manifolds}

An automorphism of an irreducible holomorphic symplectic manifold \(X\) is \textit{symplectic} if its pullback acts trivially on \(\sigma_X\). An automorphism is nonsymplectic if its pullback acts nontrivially on the  space \(H^{2,0}(X)=\mathbb{C}\sigma_X\).
A cyclic group \(G \subset \Aut(X)\) is called \textit{symplectic} if it is generated by a symplectic automorphism.

Automorphisms of irreducible holomorphic symplectic manifolds can be classified studying the induced action on \(H^{2}(X,\mathbb{Z})\) which carries a lattice structure provided by the Beauville--Bogomolov--Fujiki quadratic form. A \textit{marking} is an isometry ~ \(\eta \colon H^{2}(X,\mathbb{Z}) \rightarrow L\) where \(L\) is a lattice; the pair \((X,\eta)\) is called a \textit{marked pair}. If \((X,\eta)\) is a marked pair an isometry \(\varphi \in \bO(L)\) is \textit{symplectic} if \(\varphi \otimes \mathbb{C}\in \bO(L \otimes \mathbb{C})\) acts trivially on \(\eta(\sigma_X)\), and it is \textit{nonsymplectic} if \(\varphi \otimes \mathbb{C}\in \bO(L \otimes \mathbb{C})\) acts nontrivially on the  space \(\mathbb{C}\eta(\sigma_X)\).
A cyclic group \(G \subset \bO(L)\) is called \textit{nonsymplectic} if \(G\) is generated by a nonsymplectic isometry.

We are interested in the image of the following representation map
\begin{equation}\label{main rep map}
\eta_*\colon \Aut(X) \rightarrow \bO(L), \quad f \mapsto \eta \circ f^*\circ \eta^{-1}.
\end{equation} If an isometry \(\varphi \in \bO(L)\) and there exists \(g \in \Aut(X)\) such that \(\eta_{*}(g)=\varphi\), then \(\varphi\) is \textit{effective}. A group \(G \subset \bO(L)\) is called \textit{effective} if its elements are effective.

The aim of this paper is to study effective nonsymplectic groups \(G \subset\bO(L)\) of prime order on manifolds of \(\OG6\) type.

The global Torelli theorem for K3 surfaces, due to Piatetski-Shapiro--Shafarevich, allows us to reconstruct automorphisms of a K3 surface \(S\) starting from Hodge isometries of \(H^{2}(S,\mathbb{Z})\) which preserve the K\"{a}hler cone. Huybrechts~\cite{huybrechts2011global}, Markman~\cite{markman2011survey} and Verbitsky~\cite{verbitsky2013mapping} (see also \cite{verbitsky2013mapping.errata}) formulated similar results of Torelli type for irreducible holomorphic symplectic manifolds.

Recently Mongardi--Rapagnetta \cite{MROG6mon} computed the monodromy group for manifolds of ~ \(\OG6\) type and due to the features of this group, the global Torelli theorem \cite[Theorem 1.3]{markman2011survey} holds in a stronger form for \(\OG6\) type manifolds, namely a necessary and sufficient condition to have a bimeromorphic map between two manifolds of \(\OG6\) type is to have a Hodge isometry of the second integral cohomology.

 Classifying finite groups of automorphisms \(G\) of a certain deformation type of irreducible holomorphic symplectic manifolds can mean one of the following:
	\begin{itemize}
		\item[(1)]classifying invariant and coinvariant sublattices of the induced action of \(G\) in \(H^2(X,\mathbb{Z})\) up to isometry;
		\item[(2)] classifying the connected components of the moduli space of pairs \((X,G)\).
	\end{itemize}
	In general, classification (2) is finer than (1).
This paper deals with level \((1)\) of classification of nonsymplectic groups \(G \subset \Aut(X)\) on manifolds of ~ \(\OG6\) type.

In the case of manifolds of K3\(^{[2]}\) type the symplectic automorphisms are treated by Camere \cite{camere2012symplectic} and Mongardi \cite{mongardi2012symplectic}; the study of nonsymplectic automorphisms was started by Beauville \cite{beauville2011antisymplectic} and continued by Ohashi--Wandel \cite{ohashi2013non}, Boissière--Camere--Mongardi--Sarti \cite{BCMS16}, Boissière--Camere--Sarti \cite{boissiere2016classification}, Camere--G. Kaputska--M. Kaputska--Mongardi \cite{camere2017verra}; furthermore Boissière--Nieper-Wi\ss{}kirchen--Sarti \cite{boissiere2013smith} describe the fixed locus of these automorphisms. Camere--Cattaneo--Cattaneo \cite{Camere-Catt-Catt} study nonsymplectic involutions of ~ K3\(^{[n]}\) type manifolds and Camere--Cattaneo \cite{camere2018non} study nonsymplectic automorphisms of ~ K3\(^{[n]}\) type manifolds, where \(n \geq 3\). Moreover Joumaah \cite{joumaah2016non}, building on a work by Ohashi--Wandel \cite{ohashi2013non}, gives a criterion to find the classification (2) in the case of involutions on manifolds of ~ K3\(^{[n]}\) type.

 The study of automorphisms of generalized Kummer manifolds was started by Mongardi--Tari--Wandel \cite{MTW18} and continued by Boissi\`{e}re--Nieper-Wi\ss{}kirchen--Tari \cite{boissiere2019some} and by Brandhorst--Cattaneo \cite{brandhorst2020prime}.
 
 Recently the author together with Onorati and Veniani \cite{GOV20} classified symplectic birational transformations on manifolds of \(\OG6\) type in the case of finite cyclic groups, hence this paper completes the classification of automorphisms of manifolds of \(\OG6\) type.
The classification of nonsymplectic automorphisms on manifolds of \(\og\) type was started by Brandhorst--Cattaneo \cite{brandhorst2020prime}, and recent progress by Onorati \cite{onorati2020monodromy} about the monodromy group and the wall divisors for this deformation class constitutes a starting point for the study of the symplectic case.

\subsection{Contents of the paper} 

In \autoref{Automorphisms of irreducible holomorphic symplectic manifolds} we give a summary of basic results about irreducible holomorphic symplectic manifolds and we introduce the main tools to approach the study of automorphisms. In \autoref{preliminaries} we give basic notions of lattice theory and we recall the properties of the second integral cohomology of an irreducible holomorphic symplectic manifold. If \(X\) is an irreducible holomorphic symplectic manifold the second integral cohomology group is equipped with an integral quadratic form called \textit{Beauville--Bogomolov--Fujiki quadratic form}. 
With this form \(H^{2}(X,\mathbb{Z})\) is a lattice of signature \((3, b_2(X)-3)\).  
If~\(X\) is a manifold of \(\OG6\) type then \(H^{2}(X,\mathbb{Z})\) is isomorphic to the rank 8 and signature \((3,5)\) lattice \(\U^{\oplus 3} \oplus [-2]^{\oplus 2}\) (see \cite{rapagnetta2007topological}) and we denote this lattice by \(\bL\) throughout the paper. If a group \(G\) acts on \(\bL\) then the invariant and coinvariant sublattices are denoted by \(\bL^G\) and \(\bL_G\) respectively.

In \autoref{Invariant and coinvariant lattices in lambda} in order to obtain the classification of invariant and coinvariant lattices of \(\bL\) we classify isometries of prime order of the smallest unimodular lattice in which \(\bL\) embeds, namely \(\bLambda=\U^{\oplus5}\). 
We denote by \(\bLambda^G\) and \(\bLambda_G\) respectively the invariant and the coinvariant sublattices of \(\bLambda\) with respect to the action of a subgroup \(G \subset \bO(\bLambda)\). Taking into account \autoref{k3lattice} we determine which pairs of \(p\)-elementary lattices can represent the invariant and the coinvariant sublattices with respect to an action of a group \(G \subset \bO(\bLambda)\) of prime order \(p\).
We give the first crucial result of this paper. 

\begin{theorem}\label{thm resume lambda}
	Let \(G \subset \bO(\bLambda)\) be a subgroup of prime order \(p\).
	If either \(p = 2\) and \(\Sgn(\bLambda_G)=(2, \rk(\bLambda_G)-2)\), or \(p=2\) and \(\Sgn(\bLambda_G)=(3, \rk(\bLambda_G)-3)\), or \(p \in \{3,5,7\}\) and \(\Sgn(\bLambda_G)=(2, \rk(\bLambda_G)-2)\), then the pair \((\bLambda^G,\bLambda_G)\) appears in \autoref{tab:Lambda} on page \(11\).
\end{theorem}

 In \autoref{orders of non sympl} we determine the possible prime orders of nonsymplectic groups of isometries on manifolds of \(\OG6\) type and in \autoref{non sympl autom are effective} we find that, for manifolds of \(\OG6\) type, nonsymplectic groups \(G \subset \bO(\bL)\) of prime order are effective. In this way we obtain the classification of invariant and coinvariant sublattices of \(\bL\) with respect to effective nonsymplectic groups \(G \subset \bO(\bL)\) of prime order on a manifold of \(\OG6\) type. In \autoref{Nonsymplectic automorphisms I: a classification in dimension six} we finally come to the main result of this paper.
\begin{theorem}\label{thm resume}
	Let \(X\) be a manifold of \(\OG6\) type and let \(G \subset \bO(\bL)\) be a nonsymplectic group of prime order \(p\). Then \(G\) is an effective nonsymplectic group if and only if \(|G| \in \{2,3,5,7\}\), and the pair \((\bL^G,\bL_G)\) appears in \autoref{tab:L} on page \(17\). 
\end{theorem}
Finally we remark that in \cite{Grossi:induced} we determine if nonsymplectic  automorphisms of manifolds of \(\OG6\) type classified in \autoref{tab:L} in \autoref{thm resume} are induced \cite[Definition 3.7]{Grossi:induced}, and induced at the quotient \cite[Definition 4.2]{Grossi:induced}.

\subsection*{Acknowledgments} I wish to warmly thank Giovanni Mongardi for his supervision and for many useful hints and suggestions. Moreover I am grateful to Samuel Boissi\`{e}re, Claudio Onorati,  Alessandra Sarti, Davide Cesare Veniani, and Marc Nieper-Wi\ss{}kirchen for fruitful discussions. 
The hospitality of the University of Poitiers and the Massachusetts Institute of Technology and the financial support of G.A.H.I.A. project and INDAM are acknowledged.

\section{Preliminaries}\label{preliminaries}

In \autoref{lattices} we gather the required background and we give an overview of lattice theory and of finite quadratic forms, recalling the fundamental definitions and results which we will use throughout this paper. 

How to construct primitive embeddings of lattices is explained in \autoref{Embeddings of lattices} and in \autoref{Isometries} we recall basic results about nonsymplectic groups of isometries of lattices and \(p\)-elementary lattices.

\subsection{Lattices}\label{lattices}
Our main references for lattices are Nikulin’s paper \cite{Nik79}, but we also use \cite{conway2013sphere} and \cite[Chapter 14]{huybrechts2016lectures}. 
A lattice \(L\) is a free \(\mathbb{Z}\)-module of finite rank together with a symmetric bilinear form
\[
    ( \ , \ ) \colon L \times L \rightarrow \mathbb{Z},
\]
which we assume to be nondegenerate. A lattice \(L\) is called \textit{even} if \(x^{2} := (x, x) \in 2\mathbb{Z}\) \(\forall x \in L\). For any lattice \(L\) the \textit{discriminant group} is the finite group associated to \(L\) defined as ~ \(L^{\sharp} = L^{*}/L\), where \(L \hookrightarrow L^{*} := \Hom_{\mathbb{Z}}(L,\mathbb{Z}), \ x \longmapsto (x, \cdot)\). 
The discriminant group is a finite abelian group of order \(|{\det(L)}|\) and this number is called the \textit{discriminant} of \(L\). 
A lattice is called \textit{unimodular} if \(L^{\sharp}=\{\id\}\), and \textit{\(p\)-elementary} if the discriminant group \(L^{\sharp}\) is isomorphic to \((\mathbb{Z}/p\mathbb{Z})^{\oplus a}\) for some \(a \in \mathbb{Z}_{\geq0}\). 
The \textit{length} of the discriminant group \(L^{\sharp}\), denoted by \(l(L^{\sharp})\), is the minimal number of generators of the finite group \(L^{\sharp}\). 
The \textit{divisibility} \(\divi(v)\) of an element \(v \in L\) is the positive generator of the ideal ~ \((v,L )= \divi(v) \mathbb{Z}\). 
The pairing \(( \ , \ )\) on \(L\) induces a \(\mathbb{Q}\)-valued pairing on \(L^{*}\) and hence a pairing \(L^{\sharp} \times L^{\sharp} \rightarrow \mathbb{Q}/\mathbb{Z}\). 
If the lattice \(L\) is even, then the \(\mathbb{Q}\)-valued quadratic form on \(L^{*}\) yields
\[
    q_{L} \colon L^{\sharp} \rightarrow \mathbb{Q}/2\mathbb{Z}.
\]
The form \(q_L\) is called the \textit{discriminant quadratic form} on \(L\). There exists a natural homomorphism \(\bO(L) \rightarrow \bO(L^{\sharp})\). If \(G \subset \bO(L)\) is a subgroup of isometries then its image in \(\bO(L^{\sharp})\) is denoted by \(G^{\sharp}\). If \(g\in G \subset \bO(L)\) is an isometry, we denote by \(g^{\sharp}\) its image in \(\bO(L^{\sharp})\).

\begin{definition}\label{delta}
	Let \(L\) be a \(2\)-elementary even lattice, let \(q_{L}\) be the discriminant quadratic form on \(L\). We define 
	\[
	\delta(L)=
	\begin{cases}
	0 \ \mbox{if} \ q_L(x) \in \mathbb{Z}/2\mathbb{Z} \ \mbox{for all} \ x \in L^{\sharp} \\
	1  \ \mbox{otherwise} \\
	\end{cases}
	\]
\end{definition}

A fundamental invariant in the theory of lattices is given by the \textit{genus}.
Two lattices \(L\) and \(L'\) are in the same genus if \(L \oplus \U \cong L' \oplus \U\) or equivalently if and only if they have the same signature and discriminant quadratic form \cite[Corollary 1.9.4]{Nik79}. Sometimes we will need to know that a lattice is unique in its genus. In this paper all the lattices that we use are unique in their genus, either because they are indefinite and \(p\)-elementary \cite[Theorem 3.6.2]{Nik79}, or by an application of \cite[Theorem 1.14.2]{Nik79}. Finally if \(L\) is unique in its genus also ~ \(L(n)\) satisfies this property.

We introduce two lattices that will be useful in \autoref{Nonsymplectic automorphisms I: a classification in dimension six}:
\[
\h_{5} = \begin{pmatrix}
2 & 1 \\
1 & -2 \\
\end{pmatrix}, \ 
\K_{7} = \begin{pmatrix}
-4 & 1 \\
1 & -2 \\
\end{pmatrix},
\] 
and we recall that \(\U\) is the even unimodular lattice of rank 2 and \(\A_n\), \(\D_n\) and \(\E_n\) denote the positive definite ADE lattices. Moreover \([n]\) with \(n \in \mathbb{Z}\) denotes the rank \(1\) lattice generated by a  of square \(n\), in particular \(\A_1=[2]\).
\subsection{Embeddings of lattices}\label{Embeddings of lattices}
A morphism between two lattices \(S \rightarrow L\) is by definition a linear map that respects the quadratic forms. 
If \(S \hookrightarrow L\) has finite index then we say that \(L\) is an \textit{overlattice} of \(S\).
An injective morphism \(S \hookrightarrow L\) is called a \textit{primitive embedding} if its cokernel is torsion free and in this setting we denote by \(S^{\perp} \subset L\) the orthogonal complement.

Throughout the paper we refer to \cite[Proposition 1.15.1]{Nik79} and to \cite[Proposition 1.5.1]{Nik79} for the classification of primitive embeddings and the computation of orthogonal complements of primitive embeddings.

By \cite[Proposition 1.15.1]{Nik79} a primitive embedding \(S \hookrightarrow L\) where \(T=S^{\perp} \subset L\) is given by a subgroup \(H \subset L^{\sharp}\) which is called the \textit{embedding subgroup}, and an isometry \(\gamma \colon H \rightarrow H' \subset S^{\sharp}\) that we call the \textit{embedding isometry}. If \(\Gamma\) is the graph of \(\gamma\) in ~ \(L^{\sharp} \oplus S(-1)^{\sharp}\) then 
\[
    T^{\sharp} \cong \Gamma^{\perp}/\Gamma
\]
and 
\[
    |{\det(T)}|=|{\det (L)}| \cdot |{\det (S)}| /|H|^2.
\]
Equivalently by \cite[Proposition 1.5.1]{Nik79} we deduce that, if \(L\) is unique in its genus, a primitive embedding \(S \hookrightarrow L\) where \(T=S^{\perp} \subset L\) is given by a subgroup \(K \subset S^{\sharp}\) which is the \textit{gluing subgroup}, and an isometry \(\gamma \colon H \rightarrow H' \subset T^{\sharp}\) that we call the \textit{gluing isometry}. If \(\Gamma\) is the graph of \(\gamma\) in \(S^{\sharp} \oplus T(-1)^{\sharp}\) then
\[
    L^{\sharp} \cong \Gamma^{\perp}/\Gamma
\]
and
\[
    |{\det (L)}|=|{\det(S)}| \cdot |{\det (T)}| /|K|^2.
\]

\subsection{Isometries}\label{Isometries}
If \(L\) is a lattice and \(G \subset \bO(L)\) then the invariant sublattice of \(L\) is 
\[
L^G= \{ x \in L \mbox{ such that } g(x)=x, \forall g \in G\},
\]
and the coinvariant sublattice is 
\[
L_G = (L^{G})^{\perp}.
\]
It holds
\[
L \otimes \mathbb{Q}= (L^G \oplus L_G) \otimes \mathbb{Q}.
\]

Both the invariant and the coinvariant lattices are primitive sublattices of \(L\). In fact they can be expressed as kernels of endomorphisms. In particular, if \(G \subset \bO(L)\) is a cyclic group of order \(n\) generated by \(g\) then
\[
L^G=\Ker(g-\id), \ \ \ L_G=\Ker(\id+g+ \ldots + g^{n-1})
\]

The N\'{e}ron--Severi lattice is the \((1,1)\)--part of \(H^{2}(X,\mathbb{Z})\), i.e.\ \(\NS(X)=H^{2}(X,\mathbb{Z}) \cap H^{1,1}(X,\mathbb{C})\). The transcendental lattice \(\T(X)\) is the orthogonal complement of \(\NS(X)\) in \(H^{2}(X,\mathbb{Z})\) i.e.\ \(\T(X)=\NS(X)^{\perp}\) and it is the smallest sublattice of \(H^{2}(X,\mathbb{Z})\) such that ~ \(H^{2,0}(X)\subset \T(X) \otimes_{\mathbb{Z}} \mathbb{C}\). It holds \[H^{2}(X,\mathbb{Z}) \otimes \mathbb{Q}= (\NS(X) \oplus \T(X)) \otimes \mathbb{Q}.\]

\begin{proposition}[see, for instance, {\cite[\S 3]{Nikulin1}}]\label{inclusion} 
	If \((X,\eta)\) is a marked irreducible holomorphic symplectic manifolds with marking \(\eta \colon H^{2}(X,\mathbb{Z}) \to L\) and \(G \subset \bO(L)\) is a nonsymplectic group, then \(L^G \subset \NS(X)\) and \(\T(X) \subset L_G\).
\end{proposition}

Our convention is that the (real) spinor norm
\[
    \spin \colon \bO(L) \rightarrow \mathbb{R}^{\times}/(\mathbb{R}^{\times})^{2}=\{\pm1\}
\]
takes the value \(+1\) on a reflection \(\tau_{v}\) when \(v\) is a  such that \(v^{2}<0\). Moreover, the Cartan--Dieudonné theorem guarantees that \(\bO(L\otimes \mathbb{R})\) is generated by reflections in nonisotropic elements.
For more details about this choice we refer to \cite{Miranda.Morrison} where the opposite convention is used. We denote by \(\bO^{+}(L)\) the kernel of the spinor norm.
\begin{remark}
	An isometry \(g \in \bO(L)\) belongs to the kernel of the spinor norm if and only if \(g\) preserves the orientation of a positive definite subspace \(V  \subset L \otimes \mathbb{R}\) of maximal rank.
\end{remark}

\begin{lemma}\label{lemma spin}
	Let \(G \subset \bO(L)\) be a subgroup of order \(2\) and let \(\psi\) be a generator of \(G\). Then
	\[
	\spin(\psi)=(-1)^{s_{+}},
	\]
	where \((s_{+},s_{-})\) is the signature of \(L_G\).
\end{lemma}
\proof
Let \((t_+,t_-)\) be the signature of \(L^G\). 
Denote by \(\rho_v \in \bO(L \otimes \mathbb{R})\) the reflection with respect to the  \(v \in L \otimes \mathbb{R}\). 
Choose an orthonormal basis \(\{ e_1,...,e_{s_+},f_1,...,f_{s-}\} \) of \(L_G \otimes \mathbb{R}\) (so \(e_i^2 = 1 = -f_j^2\)), and observe that if \(v\) belongs to this basis then \(\rho_{v}\) sends \(v\) to \(-v\) and preserves both \(L^G\) and the other basis elements. Moreover \(\spin(\rho_{e_k})=-1\) and \(\spin(\rho_{f_l})=1\).
Since \(\psi\) acts as \(-\id\) on \(L_G \otimes \mathbb{R}\) and as \(\id\) on \(L^G \otimes \mathbb{R}\), we have that 
\(\psi=\rho_{e_{1}} \circ \cdots \circ \rho_{e_{s+}} \circ \rho_{f_{1}} \circ \cdots \circ \rho_{f_{s-}}\). Applying the spinor norm we have that 
\[
    \pushQED{\qed}
    \spin({\psi})=\spin(\rho_{e_{1}}) \cdots \spin(\rho_{e_{s_+}}) \spin(\rho_{f_{1}}) \cdots \spin(\rho_{f_{s-}})= (-1)^{s_{+}}. \qedhere
    \popQED
\]

\begin{lemma}[see, for instance, {\cite[\S 5.3]{boissiere2013smith}}]\label{lemma p elem}
	If \(\Lambda\) is a unimodular lattice and \(G \subset \bO(\Lambda)\) is a subgroup of prime order \(p\), then \(\Lambda_G\) and \(\Lambda^G\) are \(p\)-elementary lattices and \((\Lambda_G)^{\sharp} \cong (\Lambda^G(-1))^{\sharp}\).
\end{lemma}

\begin{lemma}[Boissiére--Nieper-Wißkirchen--Sarti{\cite[Lemma 5.3]{boissiere2013smith}} Mongardi--Tari--Wandel {\cite[Lemma 1.8]{MTW18}}]\label{Lemma su sottogruppi di ordine p}
	Let \(L\) be a lattice and \(G \subset \bO(L)\) a subgroup of prime order \(p\). Then \((p - 1) \mid \rank(L_G)\) and 
	\[
		\frac{L}{L^G \oplus  L_G} \cong (\mathbb{Z}/p\mathbb{Z})^{a}.
	\]
There are natural embeddings of \(\frac{L}{L^G  \oplus L_G}\) into the discriminant groups \((L^G)^{\sharp}\) and \((L_G)^{\sharp}\). 
Moreover, if \(m (p-1)= \rk(L_G)\) then \(a\leq m\).
\end{lemma}

\begin{proposition}[Boissi\`{e}re--Camere--Mongardi--Sarti {\cite[\S 4]{BCMS16}}]\label{discr divisibile per p alla p-2}
	Let \(L\) be a lattice with a nontrivial action of order \(p\), with rank \(p-1\), and discriminant \(d\). Then \(\frac{d}{p^{p-2}}\) is a square in ~ \(\mathbb{Q}\).
\end{proposition}

\section{Invariant and coinvariant lattices of \texorpdfstring{\(\bLambda\)}{Lambda}}\label{Invariant and coinvariant lattices in lambda}

The main goal of this paper is to classify effective nonsymplectic groups \(G \subset \bO(\bL)\) of prime order on manifolds of \(\OG6\) type.  In order to pursue this goal it is convenient to consider \textit{the} primitive embedding \(\bL \hookrightarrow \bLambda\) where \(\bLambda = \U^{\oplus5}\) is the smallest unimodular lattice in which \(\bL\) embeds. Such an embedding is unique up to isometry of \(\bLambda\) by \cite[Proposition 1.15.1]{Nik79}. In \autoref{Nonsymplectic groups of prime order are effective} we show that nonsymplectic groups \(G \subset \bO(\bL)\) of prime order are effective.
 In \autoref{Admissible signature of invariant and coinvariant sublattices of Lambda} we exhibit the possible values of the signature of \(p\)-elementary sublattices of \(\bLambda\). These values are related to the signature of the invariant and the coinvariant sublattices of \(\bL\) with respect to a nonsymplectic group \(G \subset \bO(\bL)\) of prime order \(p \in \{2,3,5,7\}\). Afterwards in \autoref{Existence of the actions of prime orders on Lambda} we give a criterion to determine if there exists a group \(G \subset \bO(\bLambda)\) of prime order such that the \(p\)-elementary sublattices of \(\bLambda\) are the invariant and the coinvariant sublattices with respect to the action of \(G\). Finally in \autoref{Proof of theorem 1.1} we prove \autoref{thm resume lambda}.
 \subsection{Nonsymplectic groups of prime order are effective} \label{Nonsymplectic groups of prime order are effective}
Consider the primitive embedding \(\bL \hookrightarrow \bLambda\).We call \(\bR=\bL^{\perp} \hookrightarrow \bLambda\) the \textit{residual lattice}. Then \(\bR \cong[2]^{\oplus2}\), thus it is a \(2\)-elementary lattice of signature \((2,0)\).
Given an isometry \(g \in \bO(\bL)\), by \cite[Corollary 1.5.2]{Nik79} there exists \(g' \in \bO(\bLambda)\) such that ~\(g'\) restricts to \(g\) on \(\bL\) if and only if there exists an isometry \(g'' \in \bO(\bR)\) with the following property: \(g^{\sharp}\) and \(g''^{\sharp}\) (see \autoref{lattices} for the notation) coincide through an isomorphism \(\bL^{\sharp} \cong \bR^{\sharp}\). In particular (cf. \cite[Lemma~2.12]{mongardi2012symplectic}) we have the following result, where we denote by \(\varphi^{\sharp}\) the image of \(\varphi\) in \(\bO(\bL^{\sharp})\).
 
 \begin{lemma}\label{come si estende l'azione se azione banale su Ax}
 	
 	Let \(X\) be a manifold of \(\OG6\) type. If \(\varphi \in \bO(\bL)\) is an isometry such that \(\varphi^{\sharp}=\id\) then there exists a primitive embedding \(\bL \hookrightarrow \bLambda\) and \(\varphi\) extends to an element \(\widetilde{\varphi} \in \bO(\bLambda)\) that acts trivially on \(\bL^{\perp} \subset \bLambda\).
 	
 \end{lemma}
 \proof
 Let \([v_{1}/2]\) and \([v_{2}/2]\) be two generators of \(\bL^{\sharp}\) such that \(v_{1}^{2}=-2\) and \(v_{2}^{2}=-2\). Then \(\varphi^{\sharp}([v_{1}/2])=[v_{1}/2]\) and \(\varphi^{\sharp}([v_{2}/2])=[v_{2}/2]\) i.e. \(\varphi(v_{1})=v_{1}+2w_{1}\) and \(\varphi(v_{2})=v_{2}+2w_{2}\), with \(w_{i} \in \bL\) for \(i\in\{1,2\}\). Consider a rank 2 lattice generated by two orthogonal elements \(r_{1}\) and \(r_{2}\) of square 2; its discriminant group is also \((\mathbb{Z}/2\mathbb{Z})^{\oplus 2}\) and it is generated by \([r_{1}/2]\) and \([r_{2}/2]\) with discriminant form given by \(q(r_{1}/2)=1/2\), \(q(r_{2}/2)=1/2\) and \((r_{1}, r_{2})=0\).
 Notice that \(\bL \oplus \mathbb{Z} r_{1} \oplus \mathbb{Z} r_{2}\) has an overlattice isometric to \(\bLambda\) which is generated by \(\bL\), \(\frac{r_1+v_1}{2}\) and \(\frac{r_2+v_2}{2}\). We extend \(\varphi\) to \(\bL \oplus \mathbb{Z}r_1 \oplus \mathbb{Z}r_2\) by imposing \(\varphi(r_1)=r_1\), \(\varphi(r_2)=r_2\) and we obtain an extension \(\widetilde{\varphi}\) of \(\varphi\) on \(\bLambda\) defined as follows:
 \[
  \widetilde{\varphi}\left(\frac{r_i+v_i}{2}\right)=\frac{r_i +\varphi(v_i)}{2}.
 \]
 \endproof
 \begin{remark}\label{TG=NS}
 	If \(X\) is an irreducible holomorphic symplectic manifold and \(G\) is a cyclic group generated by a nonsymplectic isometry, then at a generic point of the moduli space of pairs \((X,G)\) the invariant lattice is the N\'{e}ron-Severi lattice and the coinvariant one is the transcendental lattice \cite[\S3]{Nik_finitegroups}. 
 \end{remark}

 \begin{proposition}\label{orders of non sympl}
 	If \(X\) is an irreducible holomorphic symplectic manifold of \(\OG6\) type and \(G\subset \bO(\bL)\) is a nonsymplectic group of prime order \(p\) then \(p\in\{ 2,3,5,7 \}\).
 \end{proposition}


 \proof The result directly follows from \cite[Proposition 6]{Beauville_someremarks}.
 \endproof

The following result allows us to classify nonsymplectic automorphisms of prime order starting from their action on the second integral cohomology.
 
 \begin{proposition}\label{non sympl autom are effective}
 	If \(X\) is an irreducible holomorphic symplectic manifold of \(\OG6\) type and if \(G \subset \bO(\bL)\) is a nonsymplectic group of prime order \(p\) then \(G\) is effective.
 \end{proposition}
 \proof

Since \(G\) is nonsymplectic, by \autoref{inclusion} \(\T(X) \subseteq \bL_G\) and \(\bL^G \subseteq \NS(X)\). By construction a nonsymplectic group of isometries is a group of Hodge isometries. We need to check that a K\"{a}hler class is sent to a K\"{a}hler class by the elements of \(G\). 
If \(X\) admits the action of a nonsymplectic group of automorphisms then \(X\) is projective \cite[\S 4]{Beauville_someremarks}.
 Since \(\bL^G\subseteq\NS(X)\), there exists an invariant ample class. More precisely in the nonsymplectic case the signature of the invariant lattice \(\bL^G\) is \((1,\rk(\bL^G)-1)\), and the signature of the coinvariant lattice \(\bL_G\) is \((2,\rk(\bL_G)-2)\). We know by \cite[Theorem 5.4(1)]{MROG6mon} that in the \(\OG6\) case \(\Mon^{2}(X) = \bO^{+}(\bL)\).
 If \(p \neq 2\) then \(\varphi\) preserves the orientation of the positive cone since \(p\) is an odd number and \(\varphi^{p}=\id\). Then we have \(\varphi \in \bO^{+}(\bL)\cong \Mon^{2}(X)\). If \(p=2\) then \(\spin(\varphi)=(-1)^{2}=1\) by \autoref{lemma spin}, and the signature of \(\bL_G\) is \((2, \rk(\bL_G)-2)\) so \(\varphi \in \bO^{+}(\bL)=\Mon^{2}(X)\). Using \cite[Theorem 1.3]{markman2011survey} we conclude.
 \endproof


\subsection{Admissible signature of invariant and coinvariant sublattices of \texorpdfstring{\(\bLambda\)}{Lambda}}\label{Admissible signature of invariant and coinvariant sublattices of Lambda}

\begin{proposition} \label{signature of SgX and TgX trivial}
    Let \(X\) be a manifold of \(\OG6\) type and let \(G \subset \bO(\bL)\) be a subgroup of prime order \(p\). Consider a primitive embedding
	\(\bL \hookrightarrow \bLambda\), and let \(r_1\) and \(r_2\) be the generators of \(\bR=\bL^{\perp} \subset \bLambda\). Consider \(G' \subset \bO(\bLambda)\) a group of isometries such that \(G'\) restricts to \(G\) on \(\bL\). If \(|G^{\sharp}|=1\) then \(\bL_G \cong \bLambda_{G'}\) hence \(\Sgn(\bL_G)=\Sgn(\bLambda_{G'})\) and  \(\Sgn(\bL^G)=\Sgn(\bLambda^{G'})-(2,0)\).
\end{proposition}
\proof 
If \(|G^{\sharp}|=1\) then \(\bL_G \cong \bLambda_{G'}\) by \autoref{come si estende l'azione se azione banale su Ax}. 
\endproof

\begin{proposition}\label{signature of SgX and TgX nontrivial}
Let \(X\) be a manifold of \(\OG6\) type and let \(G \subset \bO(\bL)\) be a subgroup of prime order \(p\). Consider a primitive embedding
\(\bL \hookrightarrow \bLambda\), and let \(r_1\) and \(r_2\) be the generators of \(\bR=\bL^{\perp} \subset \bLambda\). Consider \(G' \subset O(\bLambda)\) a group of isometries such that \(G'\) restricts to \(G\) on \(\bL\). If \(|G^{\sharp}|=2\) then \(\bL_G \cong (r_1-r_2)^{\perp} \subset \bLambda_{G'}\) and \(\bL^G \cong (r_1+r_2)^{\perp} \subset \bLambda^{G'}\), hence \(\Sgn(\bL_G)=\Sgn(\bLambda_{G'})-(1,0)\) and  \(\Sgn(\bL^G)=\Sgn(\bLambda^{G'})-(1,0)\). 
\end{proposition}
\proof
We extend the action on \(\bLambda\) by the isometry \(\psi=\begin{pmatrix}
0 & 1 \\ 1 & 0
\end{pmatrix}\) on \(\bR\). Since \(\bLambda\) is unimodular, the gluing subgroup of \(\bL \oplus \bR \subset \bLambda\) is \(H = (\mathbb{Z}/2\mathbb{Z})^{\oplus2}\). By \cite[Proposition~1.5.1]{Nik79} there is an isometry \(\bL^{\sharp} \rightarrow \bR(-1)^{\sharp}\). Since \(G^{\sharp}\) acts on \(\bL^{\sharp}\) exchanging the generators, \(\psi^{\sharp}\) exchanges the generators of \(\bR^{\sharp}\). Since \(G'=G_{|\bL} \oplus \psi_{|\bR}\) this gives \((r_{1} + r_{2}) \in \Lambda^{G'}\) and \((r_1-r_2) \in \bLambda_{G'}\).
\endproof

\subsection{Existence of the actions of prime order on \texorpdfstring{\(\bLambda\)}{Lambda}}\label{Existence of the actions of prime orders on Lambda}

Now we give a classification of all the possible \(p\)-elementary sublattices \(S\) of \(\bLambda\), and their orthogonal complements \(T=S^{\perp} \subset \bLambda\), combining the constraints on the signature of \autoref{signature of SgX and TgX trivial} and of \autoref{signature of SgX and TgX nontrivial} with the content of \autoref{TG=NS}. 
More precisely, for \(p=2\) we look for \(2\)-elementary sublattices \(S\) of \(\bLambda\) with \(\Sgn(S)=(2, \rk(S)-2)\) and \(\Sgn(S)=(3, \rk(S)-3)\) and for \(p \in \{3,5,7\}\) we look for \(p\)-elementary sublattices \(S\) of \(\bLambda\) with \(\Sgn(S)=(2, \rk(S)-2)\). The pairs \((T,S)\) obtained in this way are candidates to be the invariant and the coinvariant sublattices with respect to \(G \subset \bO(\bLambda)\) of prime order \(p\). If \(p=2\) we know that all the ~ \(2\)-elementary lattices that we find correspond to an involution which acts as \(-\id\) on ~ \(\bLambda_G\). On the other hand, if \(p\) is an odd prime number we need the following criterion to determine if the pair \((T,S)\) corresponds to an invariant and a coinvariant sublattice respectively, with respect to \(G \subset \bO(\bLambda)\) of prime order \(p\). 
\begin{lemma}\label{lemma S si immerge nella K3}
If \(G \subset \bO(\bLambda)\) is a group of prime order \(p \in\{2,3,5,7\}\) and \(\Sgn(\bLambda_G)=(2, \rk(\bLambda_G)-2)\) then there exists a primitive embedding \(\bLambda_G \hookrightarrow H^{2}(X,\mathbb{Z})\) where \(X\) is a K3 surface.
\end{lemma}
\proof
In our assumptions the lattice \(\bLambda_G\) is a \(p\)-elementary lattice with \(\Sgn(\bLambda_G)=(2,\rk(\bLambda_G)-2)\). Since \(\bLambda_{G}\) is a sublattice of \(\bLambda\) and \(\Sgn(\bLambda)=(5,5)\), we have \(\rk(\bLambda_G) \leq 7\). The signature of the lattice \(H^{2}(X,\mathbb{Z})\) is \((3,19)\), hence we can apply \cite[Corollary 1.12.3]{Nik79} to prove that there exists a primitive embedding of \(\bLambda_G\) in the unimodular lattice \(H^2(X,\mathbb{Z})\). We have \(19-3=16\equiv 0 \ (8)\), moreover \(3-2 \geq 0\) and \(19-(\rk(\bLambda_G)-2) \geq 0\) since \((\rk(\bLambda_G)-2)\leq 5 \) by assumption. The inequality \(22-\rk(\bLambda_G) > l(\bLambda_G^{\sharp})\) is verified since \(22-\rk(\bLambda_G)\geq 15\) and \(l(\bLambda_G^{\sharp}) \leq 7\).
\endproof
\begin{definition}
Let \(X\) be a K3 surface. We denote by \(\mathbf{M}\) the isometry class of \(H^{2}(X,\mathbb{Z})\).
\end{definition}

\begin{proposition}\label{k3lattice}	Let \(p \in \{2,3,5,7\}\) and let \(S\) be a \(p\)-elementary lattice such that there exists a primitive embedding \(S \hookrightarrow \bLambda\) and such that \(\Sgn(S)=(2,\rk(S)-2)\). If there exists a group \(G \subset \bO(\bLambda)\) of prime order \(p\) such that \(S\cong \bLambda_G\) then there exists a K3 surface \(X\) with a marking \(H^{2}(X,\mathbb{Z}) \to \mathbf{M}\) and a nonsymplectic group \(G' \subset \bO(\mathbf{M})\) of prime order \(p\) such that \(S \cong \mathbf{M}_{G'}\).
	\end{proposition}
\proof
Let \(G \subset \bO(\bLambda)\) be a group of prime order \(p\) and let \(\bLambda_G\) be the coinvariant lattice, a \(p\)-elementary lattice such that \((\bLambda_G)^{\sharp} \cong (\bLambda^G)^{\sharp}\). Since the action on \((\bLambda^G)^{\sharp}\) is trivial then also the action on \((\bLambda_G)^{\sharp}\) is trivial by \cite[Corollary 1.5.2]{Nik79}. By \autoref{lemma S si immerge nella K3} there exists a primitive embedding ~\(\bLambda_G \hookrightarrow \mathbf{M}\). Recalling the construction of the gluing subgroup in \autoref{Embeddings of lattices}, since \(\mathbf{M}\) is unimodular, the gluing subgroup \(K\) coincide with \((\bLambda_G)^{\sharp}\). Since the action on \((\bLambda_G)^{\sharp}\) is trivial, then also the action on the discriminant group of the orthogonal complement \(\bLambda_G^{\perp}\subset \mathbf{M}\) is trivial due to the anti-isometry between \((\bLambda_G)^{\sharp}\) and \((\bLambda_G^{\perp})^{\sharp}\). In order to define a group \(G' \subset \bO(\mathbf{M})\) such that \(G'\) restricts to \(G\) on \(\bLambda_G\) and such that \(\bLambda_G \cong \mathbf{M}_{G'}\) we need an isometry of \(\bLambda_G^{\perp} \subset \mathbf{M}\) such that the induced action on \((\bLambda_G^{\perp})^{\sharp}\) is trivial. We show now that this isometry is \(\id_{|\bLambda_G^{\perp}}\). Indeed the action of \(G\) on \(\bLambda_G\) extends to an action on \(\mathbf{M}\) by gluing \(G_{|\bLambda_G}\) and \(\id_{|\bLambda_G^{\perp}}\) since the action of both on the gluing subgroups is trivial.
In this way we get \(\bLambda_G \cong \mathbf{M}_{G'}\) by construction. On the complex space \(\mathbf{M}\otimes \mathbb{C}\) we have a weight-two Hodge structure, where \((\mathbf{M}\otimes \mathbb{C})^{2,0}\oplus (\mathbf{M} \otimes \mathbb{C})^{0,2}=\mathbf{M}_{G'}\otimes\mathbb{C}\) is of signature \((2,\rk(S)-2)\). By Torelli theorem there exists a K3 surface \(X\) such that  \(H^{2,0}(X) \oplus H^{0,2}(X) = \mathbf{M}_{G'} \otimes \mathbb{C}\) and the action on the integral cohomology is nonsymplectic as we want.
\endproof

For \(p=5,7\) Artebani--Sarti--Taki \cite{AST11} and for \(p=3\) Artebani--Sarti \cite{AS3} find a classification of invariant and coinvariant sublattices with respect to a nonsymplectic group of prime order \(p\) on a K3 surface.
As a consequence using \autoref{k3lattice} we can determine which \(p\)-elementary sublattices \(\bLambda_G\) classified in \autoref{tab:Lambda} are the coinvariant sublattices with respect to an isometry of prime order \(p\) of \(\bLambda\). If this is the case there exists a group \(G \subset \bO(\bLambda)\) of prime order \(p\) such that \((\bLambda^G, \bLambda_G)\) are the invariant and the coinvariant sublattices. 
\subsection{Proof of \autoref{thm resume lambda}}\label{Proof of theorem 1.1}
 
In the following we prove \autoref{thm resume lambda}, dividing it into five cases.
\begin{proposition}
	If the order of \(G \subset \bO(\bLambda)\) is \(2\) and if \(\Sgn(\bLambda_G)=(2,\rk(\bLambda_G)-2)\) then there are fourteen possible pairs of invariant and coinvariant lattices \((\bLambda^G,\bLambda_G)\) with respect to the action of \(G\) on \(\bLambda\). 
\end{proposition}

\begin{proof}
	Since \(|G|=2\) then \(\bLambda^G\) and \(\bLambda_G\) are \(2\)-elementary lattices. It holds that \(\rk(\bLambda^G)=10-\rk(\bLambda_G)\), hence \(a \leq 10-\rk(\bLambda_G)\). We use \autoref{Lemma su sottogruppi di ordine p} to bound the number \(a\) that occurs there. For each possible value of \(a\) we apply \cite[Theorem 3.6.2]{Nik79} and \cite[Corollary 1.13.5]{Nik79} and we obtain the classification in \autoref{tab:Lambda}.
\end{proof}

\begin{proposition}\label{p=2 non trivial}
If the order of \(G \subset \bO(\bLambda)\) is \(2\) and if \(\Sgn(\bLambda_G)=(3,\rk(\bLambda_G)-3)\) then there are thirteen possible pairs of invariant and coinvariant lattices \((\bLambda^G,\bLambda_G)\) with respect to the action of \(G\) on \(\bLambda\).
\end{proposition}

\begin{proof}
	Recall that \(\rk(\bLambda^G)=10-\rk(\bLambda_G)\), hence \(a \leq 10-\rk(\bLambda_G)\). The result is a direct application of \autoref{Lemma su sottogruppi di ordine p}, \cite[Theorem 3.6.2]{Nik79} and \cite[Corollary 1.13.5]{Nik79}. The classification is summarized in \autoref{tab:Lambda}.
\end{proof}

\begin{proposition}	If the order of \(G \subset \bO(\bLambda)\) is \(3\) and if \(\Sgn(\bLambda_G)=(2,\rk(\bLambda_G)-2)\) then there are five possible pairs of invariant and coinvariant lattices ~ \((\bLambda^G,\bLambda_G)\) with respect to the action of \(G\) on \(\bLambda\).
\end{proposition}

\begin{proof}
	The result is a direct application of \autoref{Lemma su sottogruppi di ordine p}, the Theorem in \cite[\S1]{RS81}, \cite[Theorem 1.13.3]{Nik79}, \cite[Corollary 1.13.5]{Nik79} and \autoref{discr divisibile per p alla p-2}. The coinvariant sublattices that we find for \(p=3\) are \(\bLambda_G =\U^{\oplus2} \oplus \A_2(-1)\), \(\bLambda_G=\U \oplus \A_2(-1) \oplus \U(3)\), \(\bLambda_G=\U^{\oplus2}\), \(\bLambda_G=\U \oplus \U(3)\) and \(\bLambda_G=\A_2\) and
	due to the classification of Artebani--Sarti \cite{AS3}, they are coinvariant sublattices with respect to an isometry of order 3 on \(\bLambda\).  The classification is summarized in \autoref{tab:Lambda}. 
\end{proof}

\begin{proposition}
	If the order of \(G \subset \bO(\bLambda)\) is \(5\) and if \(\Sgn(\bLambda_G)=(2,\rk(\bLambda_G)-2)\) then there is one possible pair of invariant and coinvariant lattices \((\bLambda^G,\bLambda_G)\) with respect to the action of \(G\) on \(\bLambda\).
\end{proposition}

\begin{proof}
Recall that \(\rk(\bLambda_G)\) has to be a multiple of \(4\) and, due to the assumptions on the signature, \(\rk(\bLambda_G) \leq 7\). The result is a direct application of \autoref{Lemma su sottogruppi di ordine p}, the Theorem in \cite[\S1]{RS81}, \cite[Corollary 1.13.5]{Nik79} and \autoref{discr divisibile per p alla p-2}.  
In this case \(\bLambda_G=\U \oplus \h_{5}\) and \(\bLambda^G=\U^{\oplus 2} \oplus \h_{5}\). By \autoref{k3lattice} we check that this pair corresponds to invariant and coinvariant lattices \((\bLambda^G,\bLambda_G)\) with respect to the action of a group \(G \subset \bO(\bLambda)\) of order \(5\). The classification is summarized in \autoref{tab:Lambda}.
\end{proof}

\begin{proposition}
	If the order of \(G \subset \bO(\bLambda)\) is \(7\) and if \(\Sgn(\bLambda_G)=(2,\rk(\bLambda_G)-2)\) then there is one possible pair of invariant and coinvariant lattices \((\bLambda^G,\bLambda_G)\) with respect to the action of \(G\) on \(\bLambda\).
\end{proposition}

\begin{proof}

Recall that \(\rk(\bLambda_G)\) has to be a multiple of \(6\) and, due to the assumptions on the signature, \(\rk(\bLambda_G) \leq 7\). The result is a direct application of \autoref{Lemma su sottogruppi di ordine p}, the Theorem in \cite[\S1]{RS81}, \cite[Corollary 1.13.5]{Nik79} and \autoref{discr divisibile per p alla p-2}. We obtain a unique coinvariant lattice \(\bLambda_G =\U^{\oplus2} \oplus \K_7\) and the orthogonal complement is \(\bLambda^G=\U \oplus \K_7\).
 By \autoref{k3lattice} we check that this pair corresponds to the action of a group \(G \subset \bO(\bLambda)\) of order \(7\). The classification is summarized in \autoref{tab:Lambda}.
\end{proof}

In \autoref{tab:Lambda} in the column \(\delta\) we indicate whether the \(2\)-elementary quadratic form of the discriminant group of the lattice is integer valued, \(\delta=0\), or not, \(\delta=1\) (cf \autoref{delta}). Moreover \(a\) is the length of the discriminant group of \(\bLambda_G\) and of \(\bLambda^G\) since \(\bLambda\) is unimodular.

\begin{center}
	\begin{longtable}{lllllll}
		
		\caption{Pairs \((\bLambda^G,\bLambda_G)\) for \(G \subset \bO(\bLambda)\) of prime order \(p=2\) and \(\Sgn(\bLambda_G)=(2, \rk(\bLambda_G)-2)\), or \(p=2\) and \(\Sgn(\bLambda_G)=(3, \rk(\bLambda_G)-3)\), or \(p \in \{3,5,7\}\) and \(\Sgn(\bLambda_G)=(2, \rk(\bLambda_G)-2)\).}
	\label{tab:Lambda} \\
		
		\toprule
	 No. & \(|G|\) &  \(\bLambda_G\) & \(\bLambda^G\) & \(\Sgn(\bLambda_G)\) & \(a\) & \(\delta\)  \\

		\midrule
		\endfirsthead
		
		\multicolumn{7}{c}%
		{\tablename\ \thetable{}, follows from previous page} \\
		\midrule
	No. & \(|G|\) &  \(\bLambda_G\) & \(\bLambda^G\) & \(\Sgn(\bLambda_G)\) & \(a\) & \(\delta\)  \\
		\midrule
		\endhead
	
		\multicolumn{7}{c}{Continues on next page} \\
		\endfoot
		
		\bottomrule
		\endlastfoot

		\(1\) & \(2\) & \(\U^{\oplus 2} \oplus [-2 ] ^{\oplus 3 }\)& \([ 2 ] ^{\oplus 3 }\)& \((2,5)\) & \(3\) & \(1\) \\
		\(2\) & \(2\) & \(\U \oplus [ -2 ] ^{\oplus 3 } \oplus [2 ]\) & \([ 2 ] ^{\oplus 3} \oplus [ -2 ]\) &  \((2,4)\) & \(4\) & \(1\) \\
		\(3\) & \(2\)  & \(\U^{\oplus 2} \oplus [ -2 ] ^{\oplus 2 }\) & \(\U \oplus [ 2 ] ^{\oplus 2}\)&  \((2,4)\) & \(2\) & \(1\) \\
		\(4\) & \(2\)  & \([2]^{\oplus 2} \oplus [ -2 ] ^{\oplus 3}\) & \([ -2]^{\oplus 2} \oplus [2 ] ^{\oplus 3}\)  & \((2,3)\) & \(5\) & \(1\) \\
		\(5\) & \(2\)  & \(\U \oplus  [ -2 ] ^{\oplus 2 } \oplus [ 2 ]\) & \(\U \oplus [2 ] ^{\oplus 2 } \oplus [-2 ]\) &  \((2,3)\) & \(3\) & \(1\) \\
		\(6\) & \(2\) & \(\U^{\oplus 2} \oplus [ -2 ]\) & \(\U^{\oplus 2} \oplus [ 2 ]\) &  \((2,3)\) & \(1\) & \(1\) \\
		\(7\) & \(2\)  & \([ 2 ] ^{\oplus 2} \oplus [ -2 ] ^{\oplus 2}\) & \(\U \oplus [ 2 ] ^{\oplus 2} \oplus [ -2 ]^{\oplus 2}\) &  \((2,2)\) & \(4\) & \(1\) \\
		\(8\) & \(2\) & \(\U(2)^{\oplus 2}\)  & \(\U \oplus \U(2)^{\oplus 2}\)&  \((2,2)\) & \(4\) & \(0\) \\
		\(9\) & \(2\) & \(\U \oplus [ 2 ] \oplus[ -2 ]\) & \(\U^{\oplus 2} \oplus [ 2 ] \oplus [ -2 ]\) &  \((2,2)\) & \(2\) & \(1\) \\
		\(10\) & \(2\) & \(\U \oplus \U(2)\) & \(\U^{\oplus 2} \oplus \U(2)\) &  \((2,2)\) & \(2\) & \(0\) \\
		\(11\) & \(2\) & \(\U ^{\oplus 2}\)& \(\U^{\oplus 3}\) &  \((2,2)\) & \(0\) & \(0\) \\
		\(12\) & \(2\) & \([2 ]^{\oplus 2} \oplus [ -2 ]\) & \(\U^{\oplus 2} \oplus [-2]^{\oplus 2} \oplus [2]\) &  \((2,1)\) & \(3\) & \(1\) \\
	    \(13\) & \(2\) & \(\U \oplus [2]\) & \(\U^{\oplus 3} \oplus [ -2 ]\) &  \((2,1)\) & \(1\) & \(1\) \\
		\(14\) & \(2\) & \([ 2 ]^{ \oplus 2}\)& \(\U^{\oplus 3} \oplus [-2 ] ^{\oplus 2}\) &  \((2,0)\) & \(2\) & \(1\) \\
		\midrule
		\(1\) & \(2\) & \(\U^{\oplus 3} \oplus [ -2] ^{\oplus 2}\) & \([2]^{\oplus 2}\) & \((3,5)\) & \(2\) & \(1\) \\
		\(2\) & \(2\) & \(\U^{\oplus 2} \oplus [2] \oplus [-2]^{\oplus2}\) & \([2]^{\oplus2} \oplus [-2]\) & \((3,4)\) & \(3\) & \(1\) \\
		\(3\) & \(2\) & \(\U  \oplus [ 2 ] ^{\oplus 2 } \oplus [ -2] ^{\oplus 2}\) & \([ 2 ] ^{\oplus 2 } \oplus [ -2 ] ^{\oplus 2}\) &  \((3,3)\) & \(4\) & \(1\) \\
		\(4\) & \(2\) & \(\U \oplus \U(2)^{\oplus 2}\) & \(\U(2)^{\oplus2}\) &  \((3,3)\) & \(4\) & \(0\) \\
		\(5\) & \(2\) & \(\U^{\oplus 2} \oplus [ 2 ] \oplus [ -2 ]\) & \(\U  \oplus [ 2 ] \oplus [ -2 ]\) &  \((3,3)\) & \(2\) & \(1\) \\
		\(6\) & \(2\) & \(\U^{\oplus 2} \oplus \U(2)\) & \(\U \oplus \U(2)\) &  \((3,3)\) & \(2\) & \(0\) \\
    	\(7\) & \(2\) & \(\U ^{ \oplus 3}\)& \(\U^{\oplus 2}\) &  \((3,3)\) & \(0\) & \(0\) \\
		\(8\) & \(2\) & \([ -2 ] ^{\oplus 2} \oplus[ 2 ] ^{\oplus 3}\)& \([ 2 ] ^{\oplus 2} \oplus [ -2]^{\oplus 3}\) &  \((3,2)\) & \(5\) & \(1\) \\
		\(9\) & \(2\)  & \(\U \oplus [ -2 ]  \oplus [ 2 ] ^{\oplus 2}\) & \(\U \oplus [2 ] \oplus [ -2 ] ^{\oplus 2}\) &  \((3,2)\) & \(3\) & \(1\) \\
		\(10\) & \(2\)  & \(\U^{\oplus 2} \oplus [2] \)  & \(\U^{\oplus 2} \oplus [-2 ]\) &  \((3,2)\) & \(1\) & \(1\) \\
		\(11\) & \(2\) & \(\U(2)\oplus [2 ] ^{\oplus 2}\) & \(\U \oplus \U(2)\oplus [-2] ^{\oplus 2}\) &  \((3,1)\) & \(4\) & \(1\) \\
		\(12\) & \(2\)  & \(\U \oplus [ 2 ]^{\oplus 2}\)& \(\U^{\oplus 2} \oplus [ -2] ^{\oplus 2}\) &  \((3,1)\) & \(2\) & \(1\) \\
		\(13\) & \(2\) & \([ 2 ]^{\oplus 3}\)& \(\U^{\oplus 2} \oplus  [ -2 ] ^{\oplus 3}\) &  \((3,0)\) & \(3\) & \(1\) \\
		\midrule
		\(1\) & \(3\) & \(\U^{\oplus 2} \oplus \A_{2}(-1)\) & \(\U \oplus \A_{2}\) & \((2,4)\) & \(1\) &-- \\
		\(2\) & \(3\) & \(\U \oplus \U(3) \oplus \A_{2}(-1)\) & \(\U(3) \oplus \A_{2}\) & \((2,4)\) & \(3\) &-- \\
		\(3\) & \(3\) & \(\U^{\oplus 2} \) & \(\U^{\oplus3}\) & \((2,2)\) & \(0\) &-- \\
		\(4\) & \(3\) & \(\U \oplus \U(3) \) & \(\U^{\oplus 2} \oplus \U(3)\) & \((2,2)\) & \(2\) &-- \\
		\(5\) & \(3\)  & \(\A_{2}\) & \(\U^{\oplus 3} \oplus \A_{2}(-1)\) & \((2,0)\) & \(1\) & --  \\
		\midrule
		\(1\) & \(5\) & \(\U \oplus \h_{5}\) & \(\U^{\oplus 2} \oplus \h_{5}\) & \((2,2)\) & \(1\) & --\\
		\midrule
	    \(1\) & \(7\) & \(\U^{\oplus 2} \oplus \K_{7}\) & \(\U \oplus \K_{7}(-1)\) & \((2,4)\) & \(1\) &-- \\
	\end{longtable}
\end{center}

\section{Invariant and coinvariant lattices of \texorpdfstring{\(\bL\)}{L}}\label{Nonsymplectic automorphisms I: a classification in dimension six}

 To classify effective nonsymplectic groups of isometries \(G \subset \bO(\bL)\) means to classify invariant and coinvariant sublattices \((\bL^G,\bL_G)\) of \(\bL\). As we have already explained in \autoref{Introduction} there are two levels of classification and to classify group actions is finer than classifying pairs \((\bL^G,\bL_G)\).
The classification in \((1)\) is equivalent to counting the different primitive embeddings of \(\bL_G\) in \(\bL\). 
In \autoref{ord2trivial} and in \autoref{ord2nt} we classify pairs \((\bL^G, \bL_G)\) for effective nonsymplectic groups \(G \subset \bO(\bL)\) of order \(2\) and in \autoref{ord357} we classify pairs \((\bL^G, \bL_G)\) for effective nonsymplectic groups \(G \subset \bO(\bL)\) of order \(3, 5, 7\). In \autoref{proof of thoerem 1.2} we prove \autoref{thm resume}.

\subsection{Order 2, trivial action on the discriminant group}\label{ord2trivial}

\begin{lemma}\label{lemmabrand}
	If \(L\) is a lattice and \(G \subset \bO(L)\) is a group of order \(2\) generated by \(\varphi\), then  \(L_{\varphi}=L^{-\varphi}\).
\end{lemma}
\proof
We have \(L_{\varphi}=\Ker (\varphi +\id)=\Ker(-\varphi -\id)=\Ker((-\varphi)-\id)=L^{-\varphi}.\)
\endproof

\begin{lemma}\label{lemma simon}
	If \(G \subset \bO(\bL)\) is a group such that \(|G|=2\) and \(|G^{\sharp}|=1\) then \(\bL_G\) and \(\bL^{G}\) are \(2\)-elementary lattices.
\end{lemma}

\proof 
The lattice \(\bL_G\) is \(2\)-elementary by \autoref{come si estende l'azione se azione banale su Ax} and by \autoref{lemma p elem} using the primitive embedding \(\bL \hookrightarrow \bLambda\). If \(\varphi\) is a generator of \(G\), then
\(\varphi'= \varphi \oplus -\id_{\bL^{\perp}}\) is an extension of \(\varphi\) to \(\bLambda\) (cf \cite[Corollary 1.5.2]{Nik79}). 
The isometry \(-\varphi'=-\varphi \oplus \id_{\bL^{\perp}}\) is an isometry of order \(2\) of \(\bLambda\) hence by \autoref{lemma p elem} \(\bLambda_{-\varphi'}\) is a ~\(2\)-elementary lattice. Moreover by construction \(\bLambda_{-\varphi'} \cong \bL_{-\varphi}\) and by \autoref{lemmabrand}  \(\bL_{-\varphi} \cong \bL^{\varphi}\) so \(\bL^G \cong \bL^{\varphi}\) is \(2\)-elementary.
\endproof

\begin{proposition}\label{propsottogruppo}
	Let \(G \subset \bO(\bL)\) be a subgroup such that \(|G|=2\) and \(|G^{\sharp}|=1\). Let \(H \subset \bL^{\sharp}\) be the embedding subgroup of a primitive embedding \(\bL_G \hookrightarrow \bL\). 
	\begin{itemize}
		\item[(a)] If \(H = \{\id\}\) then \(l((\bL^G)^{\sharp})=l((\bL_G)^{\sharp})+2\).
		\item[(b)] If \(H=\mathbb{Z}/2\mathbb{Z}\) then \(l((\bL^G)^{\sharp})=l((\bL_G)^{\sharp})\).
		\item[(c)] If \(H= (\mathbb{Z}/2\mathbb{Z})^{\oplus2}\) then \(l((\bL^G)^{\sharp})= l((\bL_G)^{\sharp})-2\).
	\end{itemize}
\end{proposition}
\proof
Recall that an embedding \(\bL_{G} \hookrightarrow \bL\) is given by an isometry between a subgroup of \((\bL_G)^{\sharp}\) and a subgroup of \(\bL^{\sharp}\). 
We know that \(\bL^{\sharp} \cong (\mathbb{Z}/2\mathbb{Z})^{\oplus2}\) hence there are three choices for the the embedding subgroup \(H \subset \bL^{\sharp}\). 
By \autoref{lemma simon} the lattices \(\bL\), \(\bL_G\) and \(\bL^G\) are \(2\)-elementary, so the formula 
\[
    |{\det(\bL^G)}|=|{\det(L)}| \cdot |{\det(\bL_G)}| /|H|^2
\]
gives 
\[
    l((\bL^G)^{\sharp})=l((\bL_G)^{\sharp})+2-2n,
\]
where \(H \cong (\mathbb{Z}/2\mathbb{Z})^{\oplus n}\). Replacing \(n\) with \(0,1,2\) we obtain the statements \((a)\), \((b)\), \((c)\) respectively.
\endproof
\begin{remark}
By Proposition \cite[Theorem 3.6.2]{Nik79} if \(\delta=0\) then \(\Sgn(\bL^G)=(1,1)\) or \(\Sgn(\bL^G)=(1,5)\).
\end{remark}

Here we show how to use \autoref{propsottogruppo} to compute the only possible primitive embeddings of the coinvariant sublattices \(\bL_G=\bLambda_G\) in \(\bL\) in order to obtain the invariant sublattices \(\bL^G\). We do it in one case, the same strategy is applied for all the others. The first case of \autoref{tab:Lambda} is \(\bLambda_G=\bL_G=\U^{\oplus2} \oplus [-2]^{\oplus
 3}\). Using \autoref{propsottogruppo} we know that if the embedding subgroup \(H\) is equal to the identity then \(l((\bL^G)^{\sharp})\) is equal to \(l((\bL_G)^{\sharp})+2\), hence the lattice \(\bL^G\) has rank \(1\) and \(l((\bL^G)^{\sharp})=3+2=5\), hence this case cannot happen by \cite[Theorem 1.10.1.(2)]{Nik79}.
If the embedding subgroup \(H\) is equal to \(\mathbb{Z}/2\mathbb{Z}\) then \(l((\bL^G)^{\sharp})\) is equal to\(l((\bL_G)^{\sharp})\), hence the lattice \(\bL^G\) has rank \(1\) and \(l((\bL^G)^{\sharp})\) is equal to \(3\), hence neither this case can happen. As last choice we have the embedding subgroup \(H\) equal to\( (\mathbb{Z}/2\mathbb{Z})^{\oplus2}\) and the following relation holds: \(l((\bL^G)^{\sharp})= l((\bL_G)^{\sharp})-2\). In this case the lattice \(\bL^G\) has rank \(1\) and \(l((\bL^G)^{\sharp})\) is equal to \(1\). This case can happen and computing the embedding we get \(\bL^G=[2]\). Starting from the classification in \autoref{tab:Lambda} we classify invariant and coinvariant sublattices of \(G\) in \(\bL\) in \autoref{tab:L}.


\subsection{Order 2, nontrivial action on the discriminant group}\label{ord2nt}
Consider an embedding \(\bL \hookrightarrow \bLambda\) and a group of isometries \(G \subset \bO(\bLambda)\) such that \(|G|=2\) and \(|G^{\sharp}|=2\).
By \autoref{lemma p elem} we computed the thirteen possible isometry classes of \(\bLambda_G\) in \autoref{tab:Lambda}.

Let \(r_{1}\) and \(r_{2}\) be two orthogonal generators of \(\bR\). If \(|G^{\sharp}|=2\) then by \autoref{signature of SgX and TgX nontrivial} ~\(\bR_G \cong [4]\) is generated by the  \((r_1-r_2)\) and \(\bR^G\cong[4]\) is generated by the  \((r_1+r_2)\). By \cite[Proposition 1.15.1]{Nik79} we compute the possible primitive embeddings ~\(\bR_G\cong[4] \hookrightarrow \bLambda_G\) in order to find \(S\cong(\bR_G)^{\perp} \subset \bLambda_G\). The results of this computation are the possible isometry classes of \(S\) summarized in \autoref{tab:L_G}.

\begin{center}
	\begin{longtable}{lll}
		
		\caption{Orthogonal complements of \([4] \hookrightarrow \bLambda_G\).} 
		\label{tab:L_G}\\
		\toprule
	No. & \(\bLambda_G\) & \(S=[4]^{\perp_{\Lambda_G}}\) \\

		\midrule
		\endfirsthead
		
		\multicolumn{3}{c}%
		{\tablename\ \thetable{}, follows from previous page} \\
		\midrule
	No. & \(\bLambda_G\) & \(S=[4]^{\perp_{\Lambda_G}}\) \\
		\midrule
		\endhead

		\multicolumn{3}{c}{Continues on next page} \\
		\endfoot
		
		\bottomrule
		\endlastfoot

		\(1\) & \(\U^{\oplus3} \oplus [-2]^{\oplus2}\) & \(\U^{\oplus2} \oplus [-2]^{\oplus2} \oplus [-4]\) \\
		\midrule
		\(2\) & \(\U^{\oplus2} \oplus [2] \oplus [-2]^{\oplus2}\) & \(\U \oplus [2] \oplus \A_3(-1)\) \\
		\(3\) & \(\U^{\oplus2} \oplus [2] \oplus [-2]^{\oplus2}\) & \(\U \oplus [2] \oplus [-2]^{\oplus2} \oplus [-4]\) \\
		\midrule
		\(4\) & \(\U \oplus [2]^{\oplus2} \oplus [-2]^{\oplus2}\) & \([2]^{\oplus2} \oplus [-2]^{\oplus2} \oplus [-4]\) \\
		\(5\) & \(\U \oplus [2]^{\oplus2} \oplus [-2]^{\oplus2}\) & \(\U \oplus [-2]^{\oplus2} \oplus [4]\) \\
		\midrule
		\(6\) & \(\U \oplus \U(2)^{\oplus2}\) & \(\U(2)^{\oplus2} \oplus [-4]\) \\
		\(7\) & \(\U \oplus \U(2)^{\oplus2}\) & \(\U \oplus \U(2)\oplus [-4]\) \\
		\midrule
		\(8\) & \(\U^{\oplus2} \oplus[2] \oplus [-2]\) & \(\U  \oplus [2] \oplus [-2] \oplus [-4]\) \\
		\midrule
		\(9\) & \(\U^{\oplus 2} \oplus \U(2)\) & \(\U\oplus \U(2) \oplus [-4]\)  \\
		\(10\) & \(\U^{\oplus 2} \oplus \U(2)\) & \(\U^{\oplus2} \oplus [-4]\)  \\
		\midrule
		\(11\) & \(\U^{\oplus3}\) & \(\U^{\oplus2} \oplus [-4]\) \\
		\midrule
		\(12\) & \([2]^{\oplus3} \oplus  [-2]^{\oplus2}\) & \( [2] \oplus [-2]^{\oplus 2} \oplus  [4]\) \\
		\midrule
		\(13\) & \(\U\oplus [2]^{\oplus2}  \oplus [-2]\) & \( [2]^{\oplus2} \oplus [-2] \oplus [-4]\) \\
		\(14\) & \(\U  \oplus [2]^{\oplus2} \oplus [-2]\) & \(\U \oplus[-2] \oplus [4]\) \\
		\midrule
		\(15\) & \(\U^{\oplus2} \oplus [2]\) & \(\U \oplus[2] \oplus [-4]\) \\
		\midrule
		\(16\) & \(\U(2) \oplus [2]^{\oplus2}\) & \(\U(2) \oplus [4]\) \\
		\(17\) & \(\U(2) \oplus [2]^{\oplus 2}\) & \([2]^{\oplus2} \oplus [-4]\) \\
		\midrule
		\(18\) & \(\U \oplus [2]^{\oplus2}\) & \([2]^{\oplus2} \oplus [-4]\) \\
		\(19\) & \(\U \oplus [2]^{\oplus2}\) & \(\U \oplus [4]\) \\
		\midrule
		\(20\) & \([2]^{\oplus3}\) & \([2] \oplus [4]\) \\

	\end{longtable}	
\end{center}

Then for each \(S\) we obtain \(T= S^{\perp}\) by computing the primitive embeddings \(S \hookrightarrow \bL\) by \cite[Proposition 1.15.1]{Nik79}. All the pairs \((T,S)\) are summarized in \autoref{tab:invariante e coinvariante azione non banale su L}. Only when \(S\cong \U(2)^{\oplus2} \oplus [-4]\) and when \(S\cong \U \oplus [2] \oplus \A_3(-1)\) we do not find any primitive embedding \(S \hookrightarrow \bL\).

By \cite[Lemma 3.2]{GOV20} we know that if \(|G|=2\) and \(|G^{\sharp}|=2\) then \(|{\det (\bL_G)}|=|{\det (\bL^G)}|\). This result excludes the candidate cases 1, 2, 3, 5, 11, 13, 17, 19, 21, 23, 25, 27 of \autoref{tab:invariante e coinvariante azione non banale su L}.
\begin{lemma}\label{lemma2} In the case \(|G|=2\) and \(|G^{\sharp}|=2\), if \(|{\det (\bL_G)}|=|{\det( \bL^G)}|\) the gluing subgroup of \(\bL_G \hookrightarrow \bL\) contains no elements of order \(4\). 
\end{lemma}

\proof Let \(H\) be the gluing subgroup and suppose that it contains an element of order \(4\). If \(H'\) is the image of \(H\) in \((\bL^G)^{\sharp}\) then \(H \oplus H'\) contains an element of order \(4\) which is also the unique element of order \(4\) in \((\bL_G)^{\sharp} \oplus (\bL^G)^{\sharp}\). Hence \(H^{\perp}\oplus (H')^{\perp} \subset (\bL_G)^{\sharp} \oplus (\bL^G)^{\sharp}\) does not contain elements of order \(4\) and, in particular, it contains only elements of order two. Since \(\bL^{\sharp}\) is generated by elements in \(H^{\perp}\oplus (H')^{\perp}\) that are not in \(H \oplus H'\), \(G\) acts trivially on these elements and this is a contradiction.
\endproof

We compute the gluing subgroups for the possible \(S\) and \(T\) in \autoref{tab:invariante e coinvariante azione non banale su L} and \autoref{lemma2} excludes the cases 4, 6, 8, 9, 14, 15, 16, 18, 22. In case 25 we have \((S, T) \cong ([2]^{\oplus2}\oplus[-4],[4] \oplus \D_4(-1))\). These two pairs of lattices do not admit a gluing subgroup in \(\bL\) since the discriminant form of \(S^{\sharp}\) and the discriminant form of \(T(-1)^{\sharp}\) do not admit any isometric subgroup in both cases. In these two cases the lattices \((T,S)\) can not be invariant and coinvariant sublattices with respect to an action of a group \(G \subset \bO(\bL)\) such that \(|G|=2\) and \(|G^{\sharp}|=2\).
For the cases left i.e.\ 7, 10, 12, 20, 24, 26 we exhibit an isometry that generates a group \(G \subset \bO(\bL)\) such that ~\(|G|=2\), \(|G^{\sharp}|=2\) and it holds that \(\bL_G=S\) and \(\bL^G=T\), i.e. \((T,S)\) are invariant and coinvariant sublattices with respect to the action of \(G\) on \(\bL\). We denote these cases by \(\clubsuit\) in \autoref{tab:invariante e coinvariante azione non banale su L}. For these cases we give in \autoref{tab:H2 azione non banale su Ax} an example of ~ \(G\subset \bO(\bL)\) written with respect to the standard basis of \(\bL \cong \U^{\oplus 3} \oplus [-2]^{\oplus2}\) with ~\(|G|=2\), ~\(|G^{\sharp}|=2\) and \((\bL_G, \bL^G)\) as invariant and coinvariant sublattices \((T,S)\).

\begin{center}
	\begin{longtable}{llll}
		\caption{Orthogonal complements of \(S \hookrightarrow \bL\).}
		\label{tab:invariante e coinvariante azione non banale su L}\\
	
		\toprule
		No. & \(S\) & \(T\) & \(\clubsuit\) \\

		\midrule
		\endfirsthead
		
		\multicolumn{4}{c}%
		{\tablename\ \thetable{}, follows from previous page} \\
		\midrule
		No. & \(S\) & \(T\) & \(\clubsuit\) \\
		\midrule
		\endhead
		
		\midrule
		\multicolumn{4}{c}{Continues on next page} \\
		\endfoot
		
		\bottomrule
		\endlastfoot

		\(1\) & \(\U^{\oplus2}  \oplus [-2]^{\oplus2} \oplus [-4]\) & \([4]\) & --\\
		\(2\) & \(\U \oplus [2] \oplus [-2]^{\oplus2} \oplus [-4]\) & \([4] \oplus [-2]\) & --\\
		\(3\)  & \([2]^{\oplus2} \oplus [-2]^{\oplus2} \oplus [-4]\) & \([-2]^{\oplus 2} \oplus [4]\) & --\\
		\(4\) & \(\U \oplus [-2]^{\oplus2}  \oplus [4]\) & \([2] \oplus [-2]\oplus [-4]\) & --\\
		\(5\) & \(\U \oplus [-2]^{\oplus2}  \oplus [4]\) & \(\U\oplus [-4]\) &--\\
		\(6\) & \(\U \oplus [-2]^{\oplus2}  \oplus [4]\) & \(\U(2) \oplus [-4]\) &--\\
		\(7\) & \(\U \oplus \U(2) \oplus [-4]\) & \(\U(2) \oplus [-4]\) & \(\clubsuit\) \\
		\(8\) & \(\U \oplus \U(2) \oplus [-4]\) & \([-2]^{\oplus2} \oplus [4]\) & -- \\
		\(9\) & \(\U \oplus [2] \oplus [-2] \oplus [-4]\) & \([-2]^{\oplus 2} \oplus [4]\) & -- \\
		\(10\) & \(\U \oplus [2] \oplus [-2] \oplus [-4]\) & \([2] \oplus [-2] \oplus [-4]\) & \(\clubsuit\) \\
		\(11\) & \(\U^{\oplus2} \oplus [-4]\) & \([-2]^{\oplus2} \oplus [4]\) & -- \\
		\(12\) & \(\U^{\oplus2} \oplus [-4]\) & \(\U \oplus [-4]\) & \(\clubsuit\) \\
		\(13\) & \([2] \oplus [-2]^{\oplus 2} \oplus [4]\) & \(\U \oplus [-2] \oplus [-4]\) & -- \\
		\(14\) & \([2] \oplus [-2]^{\oplus 2} \oplus [4]\) & \([2] \oplus [-2]^{\oplus2} \oplus [-4]\) & --\\
		\(15\) & \([2]^{\oplus2} \oplus [-2]  \oplus [-4]\) & \([-2]^{\oplus3} \oplus [4]\) & --\\
		\(16\) & \([2]^{\oplus2} \oplus [-2] \oplus [-4]\) & \(\U \oplus [-2] \oplus [-4]\) & --\\
		\(17\) & \(\U \oplus[-2] \oplus [4]\) & \( [2] \oplus [-2]^{\oplus2} \oplus [-4]\) & --\\
		\(18\) & \(\U \oplus[-2] \oplus [4]\) & \(\U \oplus [-2] \oplus [-4]\) & --\\
		\(19\) & \(\U \oplus[2] \oplus [-4]\) & \([-2]^{\oplus 3} \oplus [4]\) & --\\
		\(20\) & \(\U \oplus[2] \oplus [-4]\) & \(\U \oplus [-2] \oplus [-4]\) & \(\clubsuit\) \\
		\(21\) & \(\U(2) \oplus [4]\) & \(\U(2) \oplus[-2]^{\oplus2} \oplus[-4]\) & --\\
		\(22\) & \(\U(2) \oplus [4]\) & \(\U\oplus [-2]^{\oplus2} \oplus [-4]\) & --\\
		\(23\) & \([2]^{\oplus2} \oplus [-4]\) & \([-2]^{\oplus4} \oplus [4]\) & --\\
		\(24\) & \([2]^{\oplus2} \oplus [-4]\) & \(\U\oplus [-2]^{\oplus2} \oplus [-4]\) & \(\clubsuit\) \\
		\(25\) & \(\U \oplus [4]\) & \(\U \oplus[-2]^{\oplus2} \oplus [-4]\) & --\\
		\(26\) & \(\U \oplus [4]\) & \(\U \oplus \A_{3}(-1)\) & \(\clubsuit\) \\
		\(27\) & \([2] \oplus [4]\) & \(\U \oplus [-2]^{\oplus3} \oplus[-4]\) & --\\

	\end{longtable}	
\end{center}

\begin{center}
	\begin{longtable}{llll}
			\caption{Invariant and coinvariant sublattices of nonsymplectic groups \(G \subset \bO(\bL)\) of order \(2\) and \(|G^{\sharp}|=2\) on manifolds of \(\OG6\) type.} \\
           \label{tab:H2 azione non banale su Ax} \\
	
		\toprule
		No. & \(\bL_G\) & \(\bL^G\) & example  \\

		\midrule
		\endfirsthead
		
		\multicolumn{4}{c}%
		{\tablename\ \thetable{}, follows from previous page} \\
		\midrule
		No. & \(\bL_G\) & \(\bL^G\) & example \\
		\midrule
		\endhead
		
		\midrule
		\multicolumn{4}{c}{Continues on next page} \\
		\endfoot
		
		\bottomrule
		\endlastfoot

		\(1\) & \(\U \oplus \U(2) \oplus [-4]\) & \(\U(2) \oplus [-4]\) & {\tiny \(\begin{pmatrix} -1 & 0 & 0 & 0 & 0 & 0 & 0 & 0 \\ 0 & -1 & 0 & 0 & 0 & 0 & 0 & 0 \\ 0 & 0 & 0 & 0 & 1 & 0 & 0 & 0 \\ 0 & 0 & 0 & 0 & 0 & 1 & 0 & 0 \\ 0 & 0 & 1 & 0 & 0 & 0 & 0 & 0 \\ 0 & 0 & 0 & 1 & 0 & 0 & 0 & 0 \\ 0 & 0 & 0 & 0 & 0 & 0 & 0 & 1 \\ 0 & 0 & 0 & 0 & 0 & 0 & 1 & 0 \end{pmatrix}\)} \\
		\(2\) & \(\U \oplus [2] \oplus [-2] \oplus [-4]\) & \([2] \oplus [-2] \oplus [-4]\) &   {\tiny \(\begin{pmatrix} -1 & 0 & 0 & 0 & 0 & 0 & 0 & 0 \\ 0 & -1 & 0 & 0 & 0 & 0 & 0 & 0 \\ 0 & 0 & 0 & 1 & 0 & 0 & 0 & 0 \\ 0 & 0 & 1 & 0 & 0 & 0 & 0 & 0 \\ 0 & 0 & 0 & 0 & 0 & -1 & 0 & 0 \\ 0 & 0 & 0 & 0 & -1 & 0 & 0 & 0 \\ 0 & 0 & 0 & 0 & 0 & 0 & 0 & 1 \\ 0 & 0 & 0 & 0 & 0 & 0 & 1 & 0 \end{pmatrix}\)}  \\
		\(3\) & \(\U^{\oplus2} \oplus [-4]\) & \(\U \oplus [-4]\) &  {\tiny \(\begin{pmatrix} -1 & 0 & 0 & 0 & 0 & 0 & 0 & 0 \\ 0 & -1 & 0 & 0 & 0 & 0 & 0 & 0 \\ 0 & 0 & -1 & 0 & 0 & 0 & 0 & 0 \\ 0 & 0 & 0 & -1 & 0 & 0 & 0 & 0 \\ 0 & 0 & 0 & 0 & 1 & 0 & 0 & 0 \\ 0 & 0 & 0 & 0 & 0 & 1 & 0 & 0 \\ 0 & 0 & 0 & 0 & 0 & 0 & 0 & 1 \\ 0 & 0 & 0 & 0 & 0 & 0 & 1 & 0 \end{pmatrix}\)} \\
	   \(4\) & \(\U \oplus[2] \oplus [-4]\) & \(\U \oplus [-2] \oplus [-4]\) & {\tiny \(\begin{pmatrix} 1 & 0 & 0 & 0 & 0 & 0 & 0 & 0 \\ 0 & 1 & 0 & 0 & 0 & 0 & 0 & 0 \\ 0 & 0 & -1 & 0 & 0 & 0 & 0 & 0 \\ 0 & 0 & 0 & -1 & 0 & 0 & 0 & 0 \\ 0 & 0 & 0 & 0 & 0 & -1 & 0 & 0 \\ 0 & 0 & 0 & 0 & -1 & 0 & 0 & 0 \\ 0 & 0 & 0 & 0 & 0 & 0 & 0 & 1 \\ 0 & 0 & 0 & 0 & 0 & 0 & 1 & 0 \end{pmatrix}\)} \\
	    \(5\) & \([2]^{\oplus2} \oplus [-4]\) & \(\U\oplus [-2]^{\oplus2} \oplus [-4]\) & {\tiny \(\begin{pmatrix} 1 & 0 & 0 & 0 & 0 & 0 & 0 & 0 \\ 0 & 1 & 0 & 0 & 0 & 0 & 0 & 0 \\ 0 & 0 & 0 & -1 & 0 & 0 & 0 & 0 \\ 0 & 0 & -1 & 0 & 0 & 0 & 0 & 0 \\ 0 & 0 & 0 & 0 & 0 & -1 & 0 & 0 \\ 0 & 0 & 0 & 0 & -1 & 0 & 0 & 0 \\ 0 & 0 & 0 & 0 & 0 & 0 & 0 & 1 \\ 0 & 0 & 0 & 0 & 0 & 0 & 1 & 0 \end{pmatrix}\)} \\
		\(6\) & \(\U \oplus [4]\) & \(\U \oplus \A_{3}(-1)\) &  {\tiny \(\begin{pmatrix} 1 & 0 & 0 & 0 & 0 & 0 & 0 & 0 \\ 0 & 1 & 0 & 0 & 0 & 0 & 0 & 0 \\ 0 & 0 & -1 & 0 & 0 & 0 & 0 & 0 \\ 0 & 0 & 0 & -1 & 0 & 0 & 0 & 0 \\ 0 & 0 & 0 & 0 & -1 & -2 & -1 & -1 \\ 0 & 0 & 0 & 0 & -2 & -1 & -1 & -1 \\ 0 & 0 & 0 & 0 & 2 & 2 & 2 & 1 \\ 0 & 0 & 0 & 0 & 2 & 2 & 1 & 2 \end{pmatrix}\)} \\
\end{longtable}
\end{center}	

\subsection{Order 3, 5, 7.} \label{ord357}
If \(G \subset \bO(\bL)\) and \(|G|=p\) where \(p \geq 3\) we have \(|G^{\sharp}|=1\) so by \autoref{signature of SgX and TgX trivial} we have \(\bL_G \cong \bLambda_G\) and \(\bL^G \cong \bR^{\perp} \subset \bLambda^G\). 
\begin{proposition}\label{prop p >=3}
		If \(S\) is a \(p\)-elementary lattice with \(p \in \{3,5,7\}\) and if there exists a primitive embedding \(S \hookrightarrow \bL\) then the embedding is unique up to isometry of \(\bL\).
\end{proposition}
\proof
If \(S \hookrightarrow \bL\) is a primitive embedding then the embedding subgroup \(H\) is such that \(H \subset S^{\sharp} \cong (\mathbb{Z}/p\mathbb{Z})^{\oplus a}\) and \(H \subset \bL^{\sharp} \cong (\mathbb{Z}/2\mathbb{Z})^{\oplus2}\). Since \(p \geq 3\) then \(H= \{\id\}\) hence the isometry \(\gamma=\id\) which means that \(\Gamma^{\perp}/ \Gamma \cong (S^{\perp})^{\sharp} \cong S^{\sharp} \oplus \bL^{\sharp}\) and \(q_{S^{\perp}} \cong q_{\bL} -q_{S}\). By \cite[Proposition 1.15.1]{Nik79} we conclude. 
\endproof
We know that \(\bL \hookrightarrow \bLambda\) is an embedding such that if \(G \subset \bO(\bL)\) and \(|G|=p\) with ~\(p \geq 3\) then \(|G^{\sharp}|=1\). Due to this fact it is possible to extend the action on \(\bLambda\) trivially on ~ \(\bL_G^{\perp} \cong \bL^G\) which means that  \(\bL_G \cong \bLambda_G\) is a \(p\)-elementary lattice. We compute an orthogonal complement \(\bL_G^{\perp}\) of a  primitive embedding \(\bL_G \hookrightarrow \bL\) and by \autoref{prop p >=3} all the other primitive embeddings are equivalent with respect to the level \((1)\) of the classification described in \autoref{Automorphisms of irreducible holomorphic symplectic manifolds}.

\begin{remark}
	 In cases \(|G|=3\) and \(|G|=7\) we can use \cite[Theorem 2.9]{BPV} to show that there is a unique embedding \(\bL_G \hookrightarrow \bL\) up to isometry of \(\bL\). In fact we have \(\rk(\bL_G) \leq 2\) (or \(\rk(\bL^G)\leq 2\)) and \(\U^{\oplus3} \subset \bL\). We obtain the level \((1)\) of the classification explained in \autoref{Automorphisms of irreducible holomorphic symplectic manifolds}.
\end{remark}

If \(|G|=3\) the only possible coinvariant sublattices that we find in \autoref{tab:Lambda} are ~ \(\bL_G \cong \U^{\oplus2} \oplus \A_2(-1)\), \(\bL_G\cong\U \oplus \A_{2}(-1) \oplus \U(3)\), \(\bL_G\cong\U^{\oplus2}\), \(\bL_G\cong\U \oplus \U(3)\) and \(\bL_G \cong \A_2\). By \autoref{prop p >=3} we know that the primitive embedding ~ \(\bL_G \hookrightarrow \bL\), if there exists, is unique up to isometry of \(\bL\). The embedding is given by choosing \(H=\{\id\}\) as embedding subgroup. If \(\bL_G \cong \U^{\oplus2} \oplus \A_2(-1)\) then \(\bL^G \cong [-2] \oplus [6]\). If \(\bL_G\cong\U \oplus \A_{2}(-1) \oplus \U(3)\) then there are no primitive embeddings of this lattice in \(\bL\); in fact if such an embedding exists, then the length of the discriminant group of the orthogonal complement is greater than the rank. If \(\bL_G\cong\U^{\oplus2}\) then \(\bL^{G}\cong\U \oplus [-2]^{\oplus2}\). If \(\bL_G\cong\U \oplus \U(3)\) then \(\bL^G\cong\U(3) \oplus [-2]^{\oplus 2}\) and if \(\bL_G \cong \A_2\) then ~ \(\bL^G \cong \U \oplus \A_2(-1) \oplus [-2]^{\oplus2}\).

If \(|G|=5\) the only possible coinvariant sublattice that we find in \autoref{tab:Lambda} is \(\bL_G \cong \U \oplus \h_5\). By \autoref{prop p >=3} we know that the primitive embedding \(\bL_G \hookrightarrow \bL\) is unique up to isometry of \(\bL\). The embedding is given by choosing the trivial embedding subgroup \(H=\{\id\}\) and we obtain \(\bL^G \cong [-2] \oplus [-10] \oplus \U\).

If \(|G|=7\) the only possible coinvariant sublattice that we find in \autoref{tab:Lambda} is \(\bL_G \cong \U \oplus \K_7\). By \autoref{prop p >=3} we know that the primitive embedding \(\bL_G \hookrightarrow \bL\) is unique up to isometry of \(\bL\). The embedding is given by choosing the trivial embedding subgroup \(H=\{\id\}\) and we obtain \(\bL^G \cong [-2] \oplus [14]\).

\subsection{Proof of \autoref{thm resume}}\label{proof of thoerem 1.2}
The computations in \autoref{ord2trivial}, \autoref{ord2nt}, \autoref{ord357} are the proof of \autoref{thm resume}. In fact the pairs \((\bL^G, \bL_G)\) are invariant and coinvariant sublattices with respect to effective nonsymplectic groups \(G \subset \bO(\bL)\) of prime order on a manifold \(X\) of \(\OG6\) type. Then there exist \(G'\subset \Aut(X)\) such that \(G=\eta_{*}(G')\) where \(\eta_{*}\) is the representation map recalled in equation \eqref{main rep map}. In this way nonsymplectic groups \(G \subset \Aut(X)\) of prime order on manifolds \(X\) of \(\OG6\) type are completely classified.
\endproof

\begin{center}
	\begin{longtable}{lllll}
		
	\caption{Invariant and coinvariant sublattices of effective nonsymplectic groups \(G \subset \bO(\bL)\) on manifolds of \(\OG6\) type.}
		\label{tab:L} \\
		
		\toprule
	No. & \(|G|\) & \(|G^{\sharp}|\) & \(\bL_G\) & \(\bL^G\)  \\

	\midrule
	\endfirsthead
	
	\multicolumn{5}{c}%
	{\tablename\ \thetable{}, follows from previous page} \\
	\midrule
		No. & \(|G|\) & \(|G^{\sharp}|\) & \(\bL_G\) & \(\bL^G\)  \\
	\midrule
	\endhead
	
	\midrule
	\multicolumn{5}{c}{Continues on next page} \\
	\endfoot
	
	\bottomrule
	\endlastfoot
	
	\(1\) & \(2\) & \(1\) & \(\U^{\oplus 2} \oplus [ -2 ]^{\oplus 3 }\)& \([2]\) \\
	\(2\) & \(2\)  & \(1\) & \(\U \oplus [2] \oplus [-2] ^{\oplus 3 }\) & \([2] \oplus [-2]\) \\
	\(3\) & \(2\) & \(1\) & \(\U^{\oplus 2} \oplus [-2 ] ^{\oplus 2 }\) & \(\U\) \\
	\(4\) & \(2\) & \(1\)  & \(\U^{\oplus 2} \oplus [-2 ] ^{\oplus 2 }\) & \([ 2 ] \oplus [-2]\) \\
	\(5\) & \(2\) & \(1\) & \(\U^{\oplus 2} \oplus [-2 ] ^{\oplus 2 }\) & \(\U(2)\) \\
	\(6\) & \(2\) & \(1\)  & \([2]^{\oplus2} \oplus [-2]^{\oplus3}\) & \([2] \oplus[-2]^{\oplus2}\) \\
	\(7\) & \(2\) & \(1\) & \(\U \oplus [-2] ^{\oplus 2 } \oplus [2]\) & \(\U \oplus [-2]\) \\
	\(8\) & \(2\) & \(1\) & \(\U \oplus [-2] ^{\oplus 2 } \oplus [2]\) & \([2] \oplus [-2]^{\oplus 2}\)  \\
	\(9\) & \(2\) & \(1\) & \(\U^{\oplus 2} \oplus [-2]\) & \([-2] ^{\oplus 2} \oplus [2]\) \\
	\(10\) & \(2\) & \(1\) & \(\U^{\oplus 2} \oplus [-2]\) & \(\U \oplus [-2]\) \\
	\(11\) & \(2\) & \(1\) & \([2] ^{\oplus 2} \oplus [-2] ^{\oplus 2}\) & \(\U \oplus [-2] ^{\oplus 2}\) \\
  \(12\) & \(2\) & \(1\)  & \([2] ^{\oplus 2} \oplus[-2] ^{\oplus 2}\) & \([-2] ^{\oplus 3} \oplus [2]\)  \\
	\(13\) & \(2\) & \(1\) & \(\U(2)^{\oplus 2}\)  & \(\U(2) \oplus [-2]^{\oplus2}\) \\
			\(14\) & \(2\) & \(1\)  & \(\U \oplus [2] \oplus [-2]\) & \([2] \oplus [-2]^{\oplus 3}\)  \\
			\(15\) & \(2\) & \(1\) & \(\U \oplus [2] \oplus [-2]\) & \(\U \oplus [-2] ^{\oplus 2}\) \\
			
			\(16\) & \(2\) & \(1\) & \(\U \oplus \U(2)\) & \(\U(2) \oplus [-2] ^{\oplus 2}\)  \\
			\(17\) & \(2\) & \(1\) & \(\U \oplus \U(2)\) & \(\U \oplus [-2] ^{ \oplus 2}\)  \\
			
			\(18\) & \(2\) & \(1\) & \(\U ^{\oplus 2}\)& \(\U \oplus [-2] ^{\oplus 2}\)  \\
			
			\(19\) & \(2\) & \(1\)  & \([2]^{\oplus 2} \oplus [-2]\) & \([-2]^{\oplus 4} \oplus [2]\) \\
			\(20\) & \(2\) & \(1\) & \([2]^{\oplus 2} \oplus [-2]\) & \(\U \oplus [-2]^{\oplus 3}\) \\
			
			\(21\) & \(2\) & \(1\) & \(\U \oplus [2]\) & \(\U \oplus [-2] ^{\oplus 3}\)  \\
			
			\(22\) & \(2\) & \(1\) & \([2]^{ \oplus 2}\)& \(\U \oplus[-2] ^{\oplus 4}\) \\
			\(23\) & \(2\) & \(1\) & \([2]^{ \oplus 2}\)& \(\U \oplus \D_{4}(-1)\)  \\
			\(24\) & \(2\)  & \(1\) & \([2]^{\oplus2}\) & \(\U(2) \oplus \D_{4}(-1)\)  \\
			\midrule 
	
		\(1\) & \(2\)  & \(2\) & \(\U \oplus \U(2) \oplus [-4]\) & \(\U(2) \oplus [-4]\) \\
			\(2\) & \(2\)  & \(2\) & \(\U \oplus [2] \oplus [-2] \oplus [-4]\) & \([2] \oplus [-2] \oplus [-4]\) \\
			\(3\) & \(2\)  & \(2\) & \(\U^{\oplus2} \oplus [-4]\) & \(\U \oplus [-4]\) \\
			\(4\) & \(2\)  & \(2\) & \(\U \oplus [2] \oplus [-4]\) & \(\U \oplus [-2] \oplus [-4]\) \\
			\(5\) & \(2\)  & \(2\) & \([2]^{\oplus2}\oplus [-4]\) & \(\U \oplus [-2]^{\oplus2} \oplus [-4]\) \\
			\(6\) & \(2\)  & \(2\) & \(\U \oplus [4]\) & \(\U \oplus \A_3(-1)\) \\
			\midrule
			
			\(1\) & \(3\) & \(1\) & \(\U^{\oplus2} \oplus \A_{2}(-1)\) & \([-2] \oplus [6]\) \\
			\(2\)  & \(3\) & \(1\) & \(\U^{\oplus2}\) & \(\U \oplus[-2]^{\oplus2}\) \\
			\(3\) & \(3\) & \(1\) & \(\U \oplus \U(3)\) & \(\U(3) \oplus [-2]^{\oplus2}\) \\
			\(4\) & \(3\) & \(1\) & \(\A_{2}\) & \(\U \oplus \A_{2}(-1) \oplus [-2]^{\oplus 2}\) \\
			\midrule

			\(1\) & \(5\) & \(1\) & \(\U \oplus \h_{5}\) & \([-2] \oplus [-10] \oplus \U\) \\
			\midrule
			
			\(1\) & \(7\) & \(1\) & \(\U^{\oplus 2} \oplus \K_{7}\) & \([-2] \oplus [14]\)  \\
			
	\end{longtable}
\end{center}

\bibliographystyle{amsplain}
\bibliography{Biblio}
\end{document}